\documentclass[12pt,a4,notitlepage]{amsart}
\usepackage{amsfonts,amssymb}

\newcommand{\w}{\omega}
\newcommand{\e}{\varepsilon}
\newcommand{\IR}{\mathbb R}
\newcommand{\IN}{\mathbb N}

\newcommand{\supp}{\mathrm{supp}}

\newcommand{\K}{\mathcal K}

\newcommand{\Ra}{\Rightarrow}

\newtheorem{example}{Example}

\newtheorem{theorem}{Theorem}

\newtheorem{corollary}{Corollary}
\newtheorem{lemma}{Lemma}

\theoremstyle{definition}
\newtheorem{definition}{Definition}

\newtheorem{problem}{Problem}

\usepackage{color}

\theoremstyle{definition}

\theoremstyle{remark}

\numberwithin{equation}{section}

\newcommand{\fA}{\mathfrak A}

\newcommand{\cZ}{\mathcal Z}

\newcommand{\sm}{\setminus}
\newcommand{\en}{\mathbb N}
\newcommand{\pr}{\mathrm{pr}}
\begin{document}

\title{Josefson-Nissenzweig property for $C_{p}$-spaces}
\author{T. Banakh, J. K\c akol, W. \'Sliwa}
\address{Ivan Franko National University of Lviv (Ukraine) and Jan Kochanowski University in Kielce (Poland)}

\email{t.o.banakh@gmail.com}

\address{Faculty of Mathematics and Informatics, A. Mickiewicz University  61-614 Pozna{\'{n}} , and Institute of Mathematics Czech Academy of Sciences, Prague}

\email{kakol@amu.edu.pl}
\address{Faculty of Mathematics and Natural Sciences,
University of Rzesz\'ow%
\newline
\indent%
35-310 Rzesz\'ow, Poland }
\email{sliwa@amu.edu.pl; wsliwa@ur.edu.pl}


\dedicatory{To the memory of our Friend Professor Mati Rubin}

\begin{abstract}
The famous  Rosenthal-Lacey theorem asserts that  for each infinite compact space $K$ the  Banach space $C(K)$ admits a quotient which is either a copy of $c_{0}$ or $\ell_{2}$.
The aim of the paper is to study a natural  variant of this result for the space $C_{p}(K)$ of continuous real-valued maps on $K$ with the pointwise topology.
Following famous Josefson-Nissenzweig theorem for infinite-dimensional Banach spaces we introduce a corresponding property (called Josefson-Nissenzweig property, briefly, the JNP) for $C_{p}$-spaces.
We prove: For a Tychonoff space $X$ the  space $C_p(X)$ satisfies the JNP if and only if $C_p(X)$ has a quotient isomorphic to $c_{0}$ (with the product topology of $\IR^\mathbb{N}$)  if and only if $C_{p}(X)$ contains a complemented subspace, isomorphic to $c_0$.
For a pseudocompact space $X$ the space $C_p(X)$ has the JNP if and only if $C_p(X)$ has a complemented metrizable infinite-dimensional subspace.  This applies to show that for a Tychonoff space $X$ the space $C_p(X)$ has a complemented subspace isomorphic to $\IR^\IN$ or $c_0$ if and only if $X$ is not pseudocompact or $C_p(X)$ has the JNP. The space $C_{p}(\beta\mathbb{N})$ contains a subspace isomorphic to $c_0$ and admits a  quotient isomorphic to $\ell_{\infty}$  but fails to have a quotient isomorphic to $c_{0}$. An example of a compact  space $K$ without infinite convergent sequences with $C_{p}(K)$  containing a complemented subspace isomorphic to $c_{0}$ is constructed.
\end{abstract}
\maketitle


\maketitle


\section{Introduction and  the main problem}

Let $X$ be  a Tychonoff space. By $C_{p}(X)$ we denote the space of real-valued continuous functions on $X$ endowed with the pointwise topology.

We will need  the following simple observation stating that each metrizable (linear) quotient $C_{p}(X)/Z$ of $C_p(X)$ by a closed vector subspace $Z$ of  $C_p(X)$ is\emph{ separable}. Indeed, this follows from the separability of metizable spaces of countable cellularity and the fact that  $C_p(X)$ has countable cellularity, being a dense subspace of $\IR^X$, see \cite{Arhangel}.

The classic Rosenthal-Lacey theorem, see \cite{Ro}, \cite{JR},  and \cite{La},  asserts that the Banach space $C(K)$  of continuous real-valued maps on an infinite compact space $K$ has a  quotient  isomorphic to $c_{0}$ or $\ell_{2},$ or equivalently, there exists a continuous linear (and open; by the open mapping Banach theorem) map from $C(K)$ onto $c_{0}$ or $\ell_{2}$.

This theorem motivates the following  natural question for spaces $C_{p}(X)$.

\begin{problem}\label{qu1} For which  compact  spaces $K$ any of the following equivalent conditions holds:
\begin{enumerate}
\item The space  $C_{p}(K)$ can  be mapped onto an infinite dimensional   metrizable locally convex space under a continuous open linear map.
\item The space  $C_{p}(K)$ can  be mapped onto an infinite dimensional   metrizable separable locally convex space under a continuous open linear map.
\item The space $C_{p}(K)$ has an infinite dimensional metrizable  quotient.
\item The space $C_{p}(K)$ has an infinite dimensional metrizable separable  quotient.
\item The space $C_p(K)$ has a quotient isomorphic to a dense subspace of $\IR^{\mathbb{N}}$.
\end{enumerate}
\end{problem}

  Note that there is a continuous linear map from a real topological vector space $E$ onto a dense subspace of $\IR^{\mathbb{N}}$ if and only if the continuous dual $E'$ is infinite dimensional. Thus when $K$ is infinite, (1) and (2) hold provided we delete ``open" in both cases. When we retain ``open" and delete ``metrizable" in (2), the question is unsolved and more general: \emph{For every infinite compact set $K$, does $C_p(K)$ admit an infinite dimensional separable quotient?}

In \cite{sa2} it was shown    that  $C_{p}(K)$   has an \emph{infinite-dimensional separable quotient algebra} if and only if  $K$ contains an infinite countable closed subset. Hence $C_{p}(\beta\mathbb{N})$ lacks infinite-dimensional separable quotient algebras. Nevertheless, as proved
in  \cite[Theorem 4]{kakol-sliwa}, the space  $C_{p}(K)$ has infinite-dimensional  separable  quotient for any compact space $K$ containing a copy of $\beta\mathbb{N}$.

Problem \ref{qu1} has been already partially studied in \cite{BKS}, where  we proved that for a Tychonoff space $X$ the  space $C_p(X)$ has an infinite-dimensional metrizable quotient if $X$ either contains an infinite discrete $C^*$-embedded subspace or else $X$ has a sequence $(K_n)_{n\in\IN}$ of compact subsets such that for every $n$ the space $K_n$ contains two disjoint topological copies of $K_{n+1}$. If fact, the first case (for example if compact $X$ contains a copy of $\beta\mathbb{N}$) asserts that  $C_{p}(X)$ has a quotient isomorphic  to the subspace $\ell_\infty=\{(x_n)\in \IR^\mathbb{N}:\sup_n |x_n|<\infty\}$ of $\IR^\mathbb{N}$ or to the product $\mathbb{R}^{\mathbb{N}}$.

Consequently, this theorem  reduces  Problem~\ref{qu1} to the case when $K$ is an \emph{Efimov space} (i.e. $K$ is an infinite compact space that contains neither a non-trivial  convergent sequence nor a copy of $\beta\mathbb{N}$).
Although, it is unknown if Efimov spaces exist in ZFC (see  \cite{dow}, \cite{dow1}, \cite{dow2}, \cite{efimov}, \cite{fedorczuk1},  \cite{fedorczuk2},  \cite{geschke},  \cite{hart})  we showed in \cite{kakol-sliwa} that under $\lozenge$ for some  Efimov spaces $K$ the function space $C_{p}(K)$ has an infinite dimensional metrizable quotient.

In this paper $c_{0}$ means the subspace $\{(x_n)_{n\in\mathbb N}\in \IR^\mathbb{N}:x_n\to_n 0\}$ of $\mathbb{R}^{\mathbb{N}}$ endowed with the product topology. 

\section{The main results}


For a Tychonoff space $X$ and a point $x\in X$ let $\delta_x:C_p(X)\to\IR,\,\,\, \delta_x:f\mapsto f(x),$ be the Dirac measure concentrated at $x$. The linear hull $L_p(X)$ of the set $\{\delta_x:x\in X\}$ in $\IR^{C_p(X)}$ can be identified with the dual space of $C_p(X)$.

Elements of the space $L_p(X)$ will be called {\em finitely supported sign-measures} (or simply {\em sign-measures}) on $X$.

Each $\mu\in L_p(X)$ can be uniquely written as a linear combination of Dirac measures $\mu=\sum_{x\in F}\alpha_x\delta_x$ for some finite set $F\subset X$ and some non-zero real numbers $\alpha_x$. The set $F$ is called the {\em support} of the sign-measure $\mu$ and is denoted by $\supp(\mu)$. The measure $\sum_{x\in F}|\alpha_x|\delta_x$ will be denoted by $|\mu|$ and the real number $\|\mu\|=\sum_{x\in F}|\alpha_x|$ coincides with the {\em norm} of $\mu$ (in the dual Banach space $C(\beta X)^*$).

The sign-measure $\mu=\sum_{x\in F}\alpha_x\delta_x$ determines the function $\mu:2^X\to \IR$ defined on the power-set of $X$ and assigning to each subset $A\subset X$ the real number $\sum_{x\in A\cap F}\alpha_x$. So, a finitely supported sign-measure will be considered both as a linear functional on $C_p(X)$ and an additive function on the power-set $2^X$.

The famous \emph{Josefson-Nissenzweig theorem}  asserts that for each infinite-dimensional Banach space $E$ there exists a null sequence in the weak$^{*}$-topology of the topological dual  $E^*$ of $E$ and  which is of norm one in the dual  norm, see for example \cite{Diestel}.

We propose the following corresponding property for spaces $C_{p}(X)$.
\begin{definition}\label{def}
For a   Tychonoff space $X$  the space $C_{p}(X)$ satisfies the \emph{Josefson-Nissenzweig property} (JNP in short) if there exists a sequence $(\mu_n)$ of finitely supported sign-measures  on $X$ such that $\|\mu_n\|=1$ for all $n\in \IN$, and $\mu_n(f)\to_n 0$ for each $f\in C_p(X)$.
\end{definition}

Concerning the JNP of function spaces $C_p(X)$ on compacta we have the following:
\begin{enumerate}
\item \emph{If a compact space $K$ contains a non-trivial convergent sequence, say $x_{n}\rightarrow x$, then $C_{p}(K)$ satisfies the JNP.} This is witnessed by the weak$^*$ null sequence $(\mu_n)$ of sign-measures $\mu_{n}:=\frac12(\delta_{x_{n}}-\delta_{x})$, $n\in\IN$.
\item  \emph{The space $C_{p}(\beta\mathbb{N})$ does not satisfy the JNP.} This follows directly from the Grothendieck theorem, see \cite[Corollary 4.5.8]{dales}.
\item \emph{There exists a compact space $K$ containing a copy of $\beta\mathbb{N}$ but without non-trivial convergent sequences such that $C_{p}(K)$ satisfies  the JNP,} see Example \ref{exa} below.
\end{enumerate}

Consequently,  if compact $K$ contains an infinite convergent sequence $x_{n}\rightarrow x$, then  $C_{p}(K)$ satisfies the JNP with  $C_{p}(Z)$  complemented in $C_{p}(K)$ and isomorphic to $c_{0}$, where $Z:=\{x\}\cup\{x_{n}\}_{n\in\mathbb{N}}$. However for every infinite compact  $K$ the space $C_{p}(K)$ contains a subspace isomorphic to $c_{0}$ but not necessary complemented in $C_{p}(K)$.  Nevertheless, there exists  a compact space $K$  without  infinite convergent sequences and such that $C_{p}(K)$  enjoys  the JNP  (hence  contains complemented subspaces isomorphic to $c_{0}$, as follows from Theorem \ref{c:T2} below).

 It turns out that the Josefson-Nissenzweig property characterizes an interesting case related with Problem \ref{qu1}.

\begin{theorem}\label{c:T2} For a Tychonoff space $X$ the following conditions are equivalent:
\begin{enumerate}
\item $C_{p}(X)$ satisfies the JNP;
\item $C_p(X)$ contains a complemented  subspace isomorphic to $c_0$;
\item $C_p(X)$ contains a quotient isomorphic to $c_0$.
\end{enumerate}
If the space $X$ is pseudocompact, then the conditions \textup{(1)--(3)} are equivalent to
\begin{enumerate}
\item[\textup{(4)}] $C_{p}(X)$ contains a complemented infinite-dimensional metrizable subspace;
\item[\textup{(5)}] $C_{p}(X)$ contains a complemented infinite-dimensional separable subspace;
\item[\textup{(6)}] $C_p(X)$ has an infinite-dimensional Polishable quotient.
\end{enumerate}
\end{theorem}

We recall that a locally convex space $X$  is {\em Polishable} if $X$ admits a stronger separable Fr\'echt (= complete metrizable)  locally convex topology.  Equivalently, Polishable locally convex spaces can be defined as images of separable Fr\'echet spaces under continuous linear maps. Clearly, the subspace $c_0$ of $\IR^{\IN}$ is Polishable.

A topological space $X$ is {\em pseudocompact} if it is Tychonoff and each continuous real-valued function on $X$ is bounded. It is known (see \cite{BKS}) that a Tychonoff space $X$ is not pseudocompact if and only if $C_{p}(X)$ contains a complemented copy of $\mathbb{R}^{\mathbb{N}}$. Combining this characterization with Theorem~\ref{c:T2}, we obtain another characterization related to Problem~\ref{qu1}.

\begin{corollary}\label{c01} For a Tychonoff space $X$ the following conditions are equivalent:
\begin{enumerate}
\item $C_p(X)$ has an infinite-dimensional Polishable quotient;
\item $C_p(X)$ contains a complemeneted infinite-dimensional Polishable subspace;
\item $C_p(X)$ contains a complemented subspace isomorphic to $\IR^{\IN}$ or $c_0$;
\item $X$ is not pseudocompact or $C_p(X)$ has the JNP.
\end{enumerate}
\end{corollary}

\begin{corollary}\label{co2}
The space $C_{p}(\beta\mathbb{N})$
\begin{enumerate}
\item has a quotient isomorphic to $\ell_{\infty}$;
\item contains a subspace isomorphic to $c_{0}$;
\item has no quotient isomorphic to $c_0$;
\item  has no Polishable infinite-dimensional quotients;
\item contains no complemented separable infinite-dimensional subspaces.
\end{enumerate}
\end{corollary}

Indeed, the first claim follows from \cite[Proposition]{BKS}, the others  follow from Theorem \ref{c:T2} and the statement (1) after Definition \ref{def}.

In the final Section~\ref{s:CP} we shall characterize Tychonoff spaces whose function space $C_p(X)$ is Polishable and prove the following theorem.

\begin{theorem}\label{t:CP}  For a Tychonoff space $X$ the following conditions are equivalent:
\begin{enumerate}
\item $C_p(X)$ is Polishable;
\item $C_k(X)$ is Polish;
\item $X$ is a submetrizable hemicompact $k$-space.
\end{enumerate}
\end{theorem}
In this theorem $C_k(X)$ denotes the space of continuous real-valued functions on $X$, endowed with the compact-open topology. It should be mentioned that a locally convex space is Polish if and only if it is a separable Fr\'echet space, by using, for example,   the Birkhoff-Kakutani theorem \cite[Theorem 9.1]{kechris}.

\section{Proof of Theorem \ref{c:T2}}

We start with the following

\begin{lemma}\label{main2} Let a Tychonoff space $X$ be continuously embedded into a compact Hausdorff space $K$.
 Let $(\mu_n)$ be a sequence of finitely supported sign-measures on $X$ (and so, on $K$) such that \begin{enumerate}
\item $\|\mu_n\|=1$ for all $n\in\IN$, and
\item $\mu_n(f)\to_n 0$ for all $f\in C(K)$.
\end{enumerate}
Then there exists an infinite subset $\Omega$ of $\IN$ such that
\begin{enumerate}
\item[(a)] the closed subspace $Z=\bigcap_{k\in \Omega}\{f\in C_p(X):\mu_k(f)=0\}$ of $C_p(X)$ is complemented in the subspace $L=\big\{f\in C_p(X):\lim_{k\in\Omega}\mu_k(f)=0\big\}$ of $C_p(X)$;
\item[(b)] the quotient space $L/Z$ is isomorphic to the subspace $c_0$ of $\IR^\IN$;
\item[(c)] $L$ contains a complemented subspace isomorphic to $c_0$;
\item[(d)] the quotient space $C_p(X)/Z$ is infinite-dimensional and metrizable (and so, separable).
\end{enumerate}
\end{lemma}

\begin{proof} (I) First we show that the set $M=\{\mu_n: n\in\IN\}$ in not relatively weakly compact in the dual of the Banach space $C(K)$. Indeed, assume by contrary that the closure $\overline{M}$ of $M$ in the weak topology of $C(K)^*$ is weakly compact. Applying the Eberlein-\v Smulian theorem \cite[Theorem 1.6.3]{AK}, we conclude that $\overline{M}$ is weakly sequentially compact. Thus $(\mu_n)$ has a subsequence $(\mu_{k_n})$ that weakly converges to some element $\mu_0\in C(K)^*$. Taking into account that the sequence $(\mu_n)$ converges to zero in the weak$^*$ topology of $C(K)^*$, we conclude that $\mu_0=0$ and hence $(\mu_{k_n})$ is weakly convergent to zero in $C(K)^*$.
Denote by $S$ the countable set $\bigcup_{n\in\IN}\supp(\mu_n)$.
The measures $\mu_n, n\in\IN,$ can be considered as elements of the unit sphere of the Banach space $\ell_1(S)\subset C(K)^*$.  By the Schur theorem \cite[Theorem 2.3.6]{AK}, the weakly convergent sequence $(\mu_{k_n})$ is convergent to zero in the norm topology of $\ell_1(S)$, which is not possible as $\|\mu_n\|=1$ for all $n\in\IN$. Thus the set $M$ is not relatively weakly compact in $C(K)^*$.
\medskip

\noindent (II) By the Grothendieck theorem \cite[Theorem 5.3.2]{AK} there exist a number $\epsilon >0$, a sequence $(m_n) \subset \IN$ and a sequence $(U_n)$ of pairwise disjoint open sets in $K$ such that $|\mu_{m_n}(U_n)|>\epsilon$ for any $n\in\IN$. Clearly, $\lim_{n\to\infty}\mu_k(U_n)=0$ for any $k\in\IN$, since $$\sum_{n\in\IN} |\mu_k|(U_n) = |\mu_k|(\bigcup_{n\in\IN} U_n) \leq |\mu_k|(K)=1.$$ Thus we can assume that the sequence $(m_n)$ is strictly increasing.

For some strictly increasing sequence $(n_k)\subset \IN$ we have $U_{n_k}\cap \supp(\mu_{m_{n_i}})=\emptyset$ for all $k,i\in\IN$ with $k>i$.

Put $\nu_k=\mu_{m_{n_k}}$ and $W_k=U_{n_k}$ for all $k\in\IN.$ Then
\medskip

(A1) $\nu_k(f)\to_k 0$ for every $f\in C(K);$
\smallskip

(A2) $|\nu_k(W_k)|>\epsilon$ for every $k\in\IN;$
\smallskip

(A3) $|\nu_k|(W_n)=0$ for all $k,n\in\IN$ with $k<n.$
\medskip

\noindent (III) By induction we shall construct a decreasing sequences $(N_k)$ of infinite subsets of $\IN$ with $\min N_k< \min N_{k+1}$ for $k\in\IN$ such that $|\nu_n|(W_m)\leq \epsilon/3^k$ for every $k\in \IN, m=\min N_k, n\in N_k$ and $n>m$. Let $N_0=\IN.$ Assume that for some $k\in\IN$ an infinite subset $N_{k-1}$ of $\IN$ has been constructed. Let $F$ be a finite subset of $N_{k-1}$ with $|F|>3^k/\epsilon$ and $\min F> \min N_{k-1}.$  For every $i\in F$ consider the set $$\Lambda_i=\{n\in N_{k-1}: |\nu_n|(W_n) \leq \epsilon/3^k\}.$$ For every $n\in N_{k-1}$ we get $|\nu_n|(X)\geq \sum_{i\in F} |\nu_n|(W_i).$ Hence there exists $i\in F$ such that $$|\nu_n|(W_i)\leq 1/|F| \leq \epsilon/3^k.$$ Thus $N_{k-1}=\bigcup_{i\in F} \Lambda_i,$ so for some $m\in F$ the set $\Lambda_m$ is infinite. Put $$N_k=\{n\in \Lambda_m: n>m\}\cup \{m\}.$$ Then $\min N_{k-1}< \min F \leq m=\min N_k$ and $|\nu_n|(W_m)\leq \epsilon/3^k$ for $n\in N_k$ with $n>m.$

\medskip

\noindent (IV) Let $i_k= \min N_k, \lambda_k=\nu_{i_k}$ and $V_k=W_{i_k}$ for $k\in \IN.$ Then

\smallskip

(B1) $\lambda_k(f)\to_k 0$  for every $f\in C(K);$
\smallskip

(B2) $|\lambda_k(V_k)|> \epsilon$ for every $k\in \IN$;
\smallskip

(B3) $|\lambda_k|(V_l)=0$ and $|\lambda_l|(V_k) \leq \epsilon/3^k$ for all $k,l\in \IN$ with $k<l.$
\medskip

Clearly, the set $$\Omega=\{n\in\IN:\mu_n=\lambda_k\; \mbox{for some}\; k\in\IN\}$$ is infinite. Put $$Z=\bigcap_{n\in \IN} \{f\in C_p(X): \lambda_n(f)=0\}$$ and $L=\{f\in C_p(X): \lambda_n(f)\to_n 0\}.$ Clearly, $Z$ and $L$ are subspaces of $C_p(X)$ and $Z$ is closed in $L$ and in $C_p(X)$. The linear operator $$S: L \to c_0,\; S:f\mapsto \lambda_n(f),$$ is continuous and $\ker S=Z.$

We shall construct a linear continuous map $P:c_0 \to L$ such that $S\circ P$ is the identity map on $c_0$. For every $k\in \IN$ there exists a continuous function $\varphi_k: K\to [-1,1]$ such that $$\varphi_k(s)=\lambda_k(V_k)/|\lambda_k(V_k)|$$ for $s\in V_k \cap \supp(\lambda_k)$ and $\varphi_k(s)=0$ for $s\in (K\setminus V_k).$ Then $$\lambda_k(\varphi_k)=|\lambda_k(V_k)|>\epsilon,$$ $\lambda_n(\varphi_k)=0$ for all $n,k \in \IN$ with $n<k$ and $$|\lambda_n(\varphi_k)| \leq |\lambda_n|(V_k)\leq \epsilon/ 3^k$$ for all $n,k \in \IN$ with $n>k.$
\medskip

\noindent (V) Let $(x_n)\in c_0$. Define a sequence $(x'_n)\in\IR^{\mathbb N}$ by  the recursive formula $$ x'_n:=[x_n-\sum_{1\leq k<n}x'_k \lambda_n(\varphi_k)]/\lambda_n(\varphi_n)\mbox{ \ for \ } n\in \IN.$$ We shall prove that $(x'_n)\in c_0$.

First we show that $\sup_n |x'_n|<\infty$. Since $(x_n)\in c_0$, there exists $m\in \IN$ such that $\sup_{n\geq m}|x_n|<\epsilon.$ Put $$M_n=\max \{2, \max_{1\leq k<n} |x'_k|\}$$ for $n\geq 2.$ For every $n>m$ we get $$|x'_n|=|x_n-\sum_{1\leq k<n}x'_k\lambda_n(\varphi_k)|/\lambda_n(\varphi_n)\leq [\epsilon +M_n\sum_{1\leq k<n} \epsilon/3^k]/\epsilon \leq 1+M_n/2\leq M_n.$$ Hence $M_{n+1}=\max \{M_n, |x'_n|\}\leq M_n$ for $n>m.$ Thus $d:=\sup_n |x'_n|\leq M_{m+1}<\infty.$
\smallskip

Now we show that $x'_n\to_n 0$.  Given any $\delta>0$, find $v\in \IN$ such that $d<3^v\delta$. Since $(x_n)\in c_0$ and $\lambda_n(f)\to_n 0$ for any $f\in C(K)$,  there exists $m>v$ such that for every $n\ge m$  $$|x_n|<\delta\epsilon\mbox{ \ and \ }d\sum_{1\leq k\leq v} |\lambda_n(\varphi_k)|<\delta\epsilon.$$ Then for $n\geq m$  we obtain
\begin{multline*}|x'_n|\leq \big[|x_n|+ \sum_{1\leq k\leq v} |x'_k|{\cdot}|\lambda_n(\varphi_k)| + \sum_{v< k<n} |x'_k|{\cdot}|\lambda_n(\varphi_k)|\big]/\lambda_n(\varphi_n)\leq\\
\leq\big[\delta\epsilon +\sum_{1\leq k\leq v} d{\cdot}|\lambda_n(\varphi_k)| + \sum_{v< k<n} d{\cdot}|\lambda_n(\varphi_k)|\big]/\epsilon <\delta + \delta +\sum_{v< k<n} d/3^k< 2\delta +d/3^v<3\delta.
\end{multline*}
 Thus $(x'_n)\in c_0.$
\smallskip

Clearly, the operator $$\Theta: c_0 \to c_0,\;\;\Theta:(x_n) \mapsto (x'_n),$$ is linear and continuous. We  prove that $\Theta$ is surjective. Let $(y_n)\in c_0$.  Set $t=\sup_n |y_n|.$ Let $$x_n=\sum_{k=1}^n \lambda_n (\varphi_k) y_k\mbox{ \  for \ }n\in \IN.$$ First we show that $(x_n)\in c_0.$
Given any $\delta>0$, find $v\in \IN$ with $\epsilon t<\delta 3^v$. Clearly, there exists $m>v+2$ such that $|y_n|<\delta$ and $\sum_{k=1}^v t|\lambda_n(\varphi_k)|<\delta$ for $n\geq m.$ Then for every $n\geq m$ we obtain \begin{multline*}|x_n|\leq \sum_{k=1}^n |\lambda_n(\varphi_k)|{\cdot}|y_k| \leq \sum_{k=1}^v t{\cdot}|\lambda_n(\varphi_k)| + \sum_{v<k<n} t{\cdot}|\lambda_n(\varphi_k)| + |\lambda_n(\varphi_n)|{\cdot}|y_n|<\\
<\delta + \sum_{v<k<n} t\epsilon /3^k + ||\lambda_n||{\cdot}|y_n|< \delta +t\epsilon/3^v +\delta<3\delta.
\end{multline*}
Thus $(x_n)\in c_0.$ Clearly, $\Theta ((x_n)_{n\in\mathbb N})=(y_n)_{n\in\mathbb N};$ so $\Theta$ is surjective.
\medskip

(VI) The operator $$T: c_0 \to C_p(X),\;\; T:(x_n)\mapsto \sum_{n=1}^{\infty} x_n{\cdot}\varphi_n|X,$$ is well-defined, linear and continuous, since the functions $\varphi_n, n\in \IN,$ have pairwise disjoint supports and $\varphi_n(X)\subset [-1,1], n\in \IN.$ Thus the linear operator $$\Phi=T\circ \Theta: c_0 \to C_p(X)$$ is continuous.

Let $x=(x_k) \in c_0$ and $x'=(x'_k)=\Theta(x).$ Then $$\Phi(x)=T(x')=\sum_{k=1}^{\infty} x'_k \varphi_k|X.$$ Using (B3) and the definition of $\Theta$, we get for every $n\in \IN$ $$\lambda_n(\Phi(x))=\sum_{k=1}^{\infty} x'_k\lambda_n(\varphi_k)=x'_n\lambda_n(\varphi_n)+ \sum_{1\leq k<n} x'_k \lambda_n(\varphi_k)=$$ $$(x_n-\sum_{1\leq k<n} x'_k \lambda_n(\varphi_k)) +\sum_{1\leq k<n} x'_k \lambda_n(\varphi_k)=x_n;$$ so $\lambda_n(\Phi(x))\to_n 0.$ This implies that $\Phi(x)\in L$ and $S\circ \Phi(x)=x$ for every $x\in c_0.$ Therefore the operator $P:=\Phi\circ S{\colon}L\to L$ is a continuous linear projection with $\ker P=\ker S=Z.$ Thus the subspace $Z$ is complemented in $L.$ Since $S\circ\Phi$ is the identity map on $c_0$, the map $S: L\to c_0$ is open. Indeed, let $U$ be a neighbourhood of zero in $L;$ then $V=\Phi^{-1}(U)$ is a  neighbourhood of zero in $c_0$ and $$V=S\circ \Phi (V) \subset S(U).$$ Thus the quotient space $L/Z$ is topologically isomorphic to $c_0$ and $\Phi(c_0)$ is a complemented subspace of $L$, isomorphic to $c_0$. In particular, $Z$ has infinite codimension in $L$ and in $C_p(X).$
\medskip

(VII) Finally we  prove that the quotient space $C_p(X)/Z$ is first countable and hence metrizable. Let $$U_n=\{f\in C_p(X): |f(x)|<1/n\;\mbox{for every}\;x\in \bigcup_{k=1}^n \supp (\lambda_k)\}, n\in \IN.$$ The first countability of the quotient space $C_p(X)/Z$ will follow as soon as for every neighbourhood $U$ of zero in $C_p(X)$ we find $n\in \IN$ with $Z+U_n \subset Z+U.$ Clearly we can assume that $$U=\bigcap_{x\in F}\{f\in C_p(X): |f(x)|<\delta\}$$ for some finite subset $F$ of $X$ and some $\delta >0.$

By the continuity of the operator $\Phi: c_0 \to C_p(X),$ there exists $n\in \IN$ such that for any $y=(y_k)\in c_0$ with $$\max_{1\leq k \leq n}|y_k| \leq 1/n$$ we get $\Phi(y)\in \frac12U$. Replacing $n$ by a larger number, we can assume that $\frac1n<\frac12\delta$ and $$F\cap \bigcup_{k=1}^{\infty}\supp(\lambda_k) \subset \bigcup_{k=1}^n \supp(\lambda_k).$$ Let $f\in U_n.$ Choose a function $h\in C_p(K)$ such that $h(x)=f(x)$ for every $$x\in F\setminus \bigcup_{k=1}^{\infty}\supp(\lambda_k)$$ and $h(x)=0$ for every $x\in \bigcup_{k=1}^n \supp (\lambda_k)$. Put $g=h|X.$ Then $g\in L$, since $\lambda_k (g)=\lambda_k (h)\to_k 0$. Put $y=S(g)$ and $\xi=\Phi(y).$ Since $g(x)=0$ for $$x\in \bigcup_{k=1}^n \supp \lambda_k,$$ we have $|\lambda_k(g)|=0<\frac1n$ for $1\leq k \leq n$, so $\max_{1\leq k \leq n} |y_k|< \frac1n$. Hence $\xi=\Phi(y) \in \frac12U,$ so $\max_{x\in F} |\xi (x)|<\frac12\delta.$ For $\varsigma=g-\xi$ we obtain $$S(\varsigma)=S(g)-S\circ \Phi \circ S (g)= S(g)-S(g)=0,$$ so $\varsigma \in Z.$ Moreover $f-\varsigma \in U.$ Indeed, we have $$|f(x)-\varsigma (x)|=|f(x)-g(x)+\xi (x)|=|\xi (x)|<\delta$$ for $x\in F\setminus \bigcup_{k=1}^{\infty}\supp (\lambda_k)$ and $$|f(x)-\varsigma (x)|=|f(x)-g(x)+\xi (x)|\leq |f(x)|+|g(x)|+|\xi (x)|<\tfrac1n+0+\tfrac12\delta<\delta$$ for every $x\in \bigcup_{k=1}^n \supp (\lambda_k)$. Thus $f=\varsigma + (f-\varsigma) \in Z+U,$ so $U_n \subset Z+U.$ Hence $Z+U_n \subset Z+U.$
\end{proof}

\begin{lemma}\label{l:MS} Let $X$ be a Tychonoff space. Each metrizable continuous image of $C_p(X)$ is separable.
\end{lemma}

\begin{proof}  It is well-known \cite[2.3.18]{Eng} that the Tychonoff product $\IR^X$ has {\em countable cellularity}, which means that $\IR^X$ contains no uncountable family of pairwise disjoint non-empty open sets. Then the dense subspace $C_p(X)$ of $\IR^X$ also has countable cellularity and so does any continuous image $Y$ of $C_p(X)$. If $Y$ is metrizable, then $Y$ is separable according to Theorem 4.1.15 in \cite{Eng}.
\end{proof}

\begin{lemma}\label{l:Polish} Let $X$ be a pseudocompact space. A closed linear subspace $S$ of $C_p(X)$ is separable if and only if $S$ is Polishable.
\end{lemma}

\begin{proof} If $S$ is Polishable, then $S$ is separable, being a continuous image of a separable Fr\'echet locally convex space. Now assume that $S$ is separable.
Fix a countable dense subset $\{f_n\}_{n\in\mathbb N}$ in $S$ and consider the continuous map
$$f:X\to\IR^\IN,\;\;f:x\mapsto (f_n(x))_{n\in\IN}.$$
By the pseudocompactness of $X$ and the metrizability of $\IR^\IN$, the image $M:=f(X)$ is a compact metrizable space. The continuous surjective map $f:X\to M$ induces an isomorphic embedding $$C_pf:C_p(M)\to C_p(X),\;\;C_pf:\varphi\mapsto \varphi\circ f.$$ So, we can identify the space $C_p(M)$ with its image $C_pf(C_p(M))$ in $C_p(X)$. We claim that $C_p(M)$ is closed in $C_p(X)$. Given any function $\varphi\in C_p(X)\setminus C_p(M)$, we should find a neighborhood $O_\varphi\subset C_p(X)$ of $\varphi$, which is disjoint with $C_p(M)$.

We claim that there exist points $x,y\in X$ such that $f(x)=f(y)$ and $\varphi(x)\ne\varphi(y)$. In the opposite case, $\varphi=\psi\circ f$ for some bounded function $\psi:M\to\IR$. Let us show that the function $\psi$ is continuous. Consider the continuous map $$h:X\to M\times \IR,\,\,\, h:x\mapsto (f(x),\varphi(x)).$$ The pseudocompactness of $X$ implies that the image $h(X)\subset M\times\IR$ is a compact closed subset of $M\times\IR$.
Let $\pr_M:h(X)\to M$ and $\pr_\IR:h(X)\to\IR$ be the coordinate projections.
It follows that $$\pr_\IR\circ h=\varphi=\psi\circ f=\psi\circ \pr_M\circ h,$$ which implies that $\pr_\IR=\psi\circ \pr_X$. The map $\pr_M:h(X)\to M$ between the compact metrizable spaces $h(X)$ and $M$ is closed and hence is quotient. Then the continuity of the map $\pr_\IR=\psi\circ\pr_X$ implies the continuity of $\psi$. Now we see that the function $\varphi=\psi\circ f$ belongs to the subspace $C_p(M)\subset C_p(X)$, which contradicts the choice of $\varphi$. This contradiction shows that $\varphi(x)\ne\varphi(y)$ for some points $x,y\in X$ with $f(x)=f(y)$.
Then $$O_\varphi:=\{\phi\in C_p(X):\phi(x)\ne\phi(y)\}$$ is a required neighborhood of $\varphi$, disjoint with $C_p(M)$.

Therefore the subspace $C_p(M)$ of $C_p(X)$ is closed and hence $C_p(M)$ contains the closure $S$ of the dense set $\{f_n\}_{n\in\mathbb N}$ in $S$. Since the space $C_p(M)$ is Polishable, so is its closed subspace $S$.
\end{proof}

Now we are at the position to prove the main Theorem \ref{c:T2}:
\smallskip
\begin{proof}[Proof of Theorem \ref{c:T2}] First, for a Tychonoff space $X$ we prove the equivalence of the conditions:
\begin{enumerate}
\item $C_{p}(X)$ satisfies the JNP;
\item  $C_p(X)$ contains a complemented  subspace isomorphic to $c_0$;
\item $C_{p}(X)$ has a quotient isomorphic to $c_0$.
\end{enumerate}

The implication $(1)\Ra(2)$ follows from Lemma~\ref{main2}, applied to the Stone-\v Cech compactification $K=\beta X$ of $X$. The implication $(2)\Ra(3)$ is trivial.
\smallskip

 To prove the implication $(3)\Ra(1)$, assume that $C_p(X)$ has a quotient isomorphic to $c_0$. Then it admits an open continuous linear operator $T:C_p(X)\to c_0$. Let $$\{e^*_n\}_{n\in\mathbb{N}}\subset c_0^*$$ be the sequence of coordinate functional. By definition of $c_0$, $e^*_n(y)\to_n 0$ for every $y\in c_0$. 
 For every $n\in\IN$ consider the linear continuous functional $$\lambda_n\in C_p(X)^*,\;\;\lambda_n:f\mapsto \lambda_n(f),$$ which can be thought as a finitely supported sign-measure on $X$. It follows that for every $f\in C_p(X)$ we have $\lambda_n(f)=e^*_n(Tf)\to_n 0$. If $\|\lambda_n\|\not\to_n 0$, then we can find an infinite subset $\Omega\subset\mathbb{N}$ such that $\inf_{n\in\Omega}\|\lambda_n\|>0$. For every $n\in\Omega$ put $$\mu_n:=\frac{\lambda_n}{\|\lambda_n\|}\in C_p(X)^*$$ and observe that the sequence $\{\mu_n\}_{n\in\IN}$ witnesses that the function space $C_p(X)$ has the JNP.

It remains to consider the case when $\|\lambda_n\|\to_n 0$. We are going to  prove that the assumption $\|\lambda_n\|\to_n 0$ leads to a  contradiction.

First we show that the union $S:=\bigcup_{n\in\IN}\supp(\lambda_n)$ is bounded in $X$ in the sense that for any continuous function $\varphi:X\to[0,+\infty]$ the image $\varphi(S)$ is bounded in $\IR$. To derive a contradiction, assume that for some function $\varphi\in C_p(X)$ the image $\varphi(S)$ is unbounded. Then we can find an increasing number sequence $(n_k)_{k\in\mathbb{N}}$ such that $$\max \varphi(\supp(\lambda_{n_k}))>3+\max \varphi(\supp(\lambda_{n_i}))\mbox{ \ \  for any \ \ }i<k.$$
For every $k\in\mathbb{N}$ choose a point $x_k\in\supp(\lambda_{n_k})$ with $$\varphi(x_k)=\max\varphi(\supp(\lambda_{n_k})).$$ It follows that $\varphi(x_k)>3+\varphi(x_i)$ for every $i<k$. Since the space $X$ is Tychonoff, for every $k\in\IN$ we can find an open neighborhood $U_k\subset \{x\in X:|\varphi(x)-\varphi(x_k)|<1\}$ of $x_k$ such that $U_k\cap \supp(\lambda_{n_k})=\{x_k\}$. Also find a continuous function $\psi_k:X\to[0,1]$ such that $\psi_k(x_k)=1$ and $\psi_k(X\setminus U_k)\subset\{0\}$.

Inductively, choose a sequence of positive real numbers $(r_k)_{k\in\IN}$ such that for every $k\in\IN$
$$r_k{\cdot}|\lambda_{n_k}(\{x_k\})|>1+\sum_{i<k}\sum_{x\in \supp(\lambda_{n_k})\setminus\{x_k\}}r_i{\cdot}\psi_i(x){\cdot}|\lambda_{n_k}(\{x\})|.$$
Since the family $(U_k)_{k\in\mathbb{N}}$ is discrete,  the function $$\psi:X\to\IR, \,\,\,\psi:x\mapsto \sum_{k=1}^\infty r_k\psi_k(x)$$ is well-defined and continuous. It follows that for every $k\in\IN$ and $i>k$ we have $U_i\cap\supp(\lambda_{n_k})=\emptyset$ and hence
\begin{multline*}
|\lambda_{n_k}(\psi)|=|\sum_{i\le k}r_i{\cdot}\lambda_{n_k}(\psi_i)|\ge r_k{\cdot}|\lambda_{n_k}(\psi_k)|-\sum_{i<k}r_i{\cdot}|\lambda_{n_k}(\psi_i)|\ge\\
\ge r_k{\cdot}|\lambda_{n_k}(\{x_k\})|-\sum_{i<k}\sum_{x\in\supp(\lambda_{n_k})\setminus\{x_k\}}r_i{\cdot}\psi_i(x)\cdot|\lambda_{n_k}(\{x\})|>1.
\end{multline*}
But this contradicts $\lambda_n(\psi)\to_n 0$. This contradiction shows that the set $S=\bigcup_{k\in\IN}\supp(\lambda_k)$ is bounded in $X$ and so is its closure $\bar S$ in $X$.

Consider the space $$C_p(X{\restriction}\bar S)=\{f{\restriction}\bar S:f\in C_p(X)\}\subset \IR^{\bar S}$$ and observe that the restriction operator $R:C_p(X)\to C_p(X{\restriction}\bar S;X)$, $R:f\mapsto f{\restriction}\bar S$, is continuous and open. Each sign-measure $\lambda_n$ has support $$\supp(\lambda_n)\subset S\subset \bar S$$ and hence can be considered as a linear continuous functional on $C_p(X{\restriction}\bar S)$. Then the operator $\tilde T:C_p(X{\restriction}S)\to c_0$, $\tilde T:f\mapsto (\lambda_n(f))_{n\in\IN}$, is well-defined and linear. Since $T=\tilde T\circ R$ is open and continuous, so is the operator $\tilde T:C_p(X{\restriction}\bar S)\to c_0$. Let $C(X{\restriction}\bar S)$ denote the space $C_p(X{\restriction}\bar S)$ endowed with the sup-norm $$\|f\|_\infty=\sup_{x\in \bar S}|f(x)|.$$ This norm is well-defined since the set $\bar S$ is bounded in $X$. The completion $\overline{C(X{\restriction}\bar S)}$ of the normed space $C(X{\restriction}\bar S)$ can be identified with a closed subspace of the Banach space $C_b(\bar S)$ of bounded continuous functions on $\bar S$, endowed with the sup-norm. It follows from $\|\lambda_n\|\to_n 0$ and $\lambda_n(f)\to_n 0$ for all $f\in C(X{\restriction}\bar S)$ that $\lambda_n(f)\to_n 0$ for all $$f\in  \overline{C(X{\restriction}\bar S)}\subset C_b(\bar S).$$

So, $\Lambda:\overline{C(X{\restriction}\bar S)}\to c_0$, $\Lambda:f\mapsto(\lambda_n(f))_{n\in\IN},$ is a well-defined continuous operator such that $T=\Lambda\circ R$. It follows that the operator $$\Lambda: \overline{C(X{\restriction}\bar S)}\to (c_0,\|\cdot\|)$$ to $c_0$ endowed with its standard norm $\|x\|=\sup_{n\in\IN}|e_n(x)|$ has closed graph and hence is continuous and open (being surjective). Then the image $\Lambda(B_1)$ of the unit ball $$B_1=\{f\in \overline{C(X{\restriction}\bar S)}:\|f\|_\infty<1\}$$ contains some closed $\e$-ball $B_\e:=\{x\in c_0:\|x\|\le\e\}$  in the Banach space $(c_0,\|\cdot\|)$. Since $\|\lambda_n\|\to_n 0$, we can find $n\in\IN$ such that $\|\lambda_n\|<\e$. Next, find an element $y\in B_\e\subset c_0$ such that $\|y\|=e^*_n(y)=\e$. Since $y\in B_\e\subset\Lambda(B_1)$, there exists a point $x\in B_1$ such that $\Lambda(x)=y$. Then $$\e=e^*_n(y)=\lambda_n(x)\le\|\lambda_n\|{\cdot}\|x\|<\e$$ and this contradiction completes the proof of the implication $(3)\Ra(1)$.
\smallskip

Now assuming that the space $X$ is pseudocompact, we shall prove that the conditions $(1)$--$(3)$ are equivalent to
\begin{enumerate}
\item[(4)] $C_p(X)$ contains a complemented infinite-dimensional metrizable subspace;
\item[(5)] $C_{p}(X)$ contains a complemented infinite-dimensional separable subspace;
\item[(6)] $C_p(X)$ has an infinite-dimensional Polishable quotient.
\end{enumerate}
\smallskip

It suffices to prove the implications $(2)\Ra(4)\Ra(5)\Ra(6)\Ra(1)$.
The implication $(2)\Ra(4)$ is trivial and $(4)\Ra(5)\Ra(6)$ follow from Lemmas~\ref{l:MS} and \ref{l:Polish}, respectively.
\smallskip

$(6) \Rightarrow (1)$: Assume that the space $C_p(X)$ contains a closed subspace $Z$ of infinite codimension such that the quotient space $E:=C_p(X)/Z$ is Polishable. Denote by $\tau_p$ the quotient topology of $C_p(X)/Z$ and by $\tau_0\supset\tau_p$ a stronger separable Fr\'echet locally convex topology on $E$.
Denote by $\tau_\infty$ the topology of the quotient Banach space $C(X)/Z$. Here  $C(X)$ is endowed with the sup-norm $\|f\|_\infty:=\sup_{x\in X}|f(x)|$ (which is well-defined as $X$ is pseudocompact).

 The identity maps between $(E, \tau_0)$ and $(E, \tau_{\infty})$ have closed graphs, since $\tau_p \subset \tau_0 \cap \tau_{\infty}.$ Using the Closed Graph Theorem we infer that the topologies $\tau_0$ and $\tau_{\infty}$ are equal. Let $G$ be a countable subset of $C(X)$ such that the set $\{g+Z: g\in G\}$ is dense in the Banach space $C(X)/Z$. Then the set $$G+Z=\{g+z:g\in G,\; z\in Z\}$$ is dense in $C(X)$. Let $(g_n)_{n\in\mathbb N}$ be a linearly independent sequence in $G$ such that its linear span $G_0$ has $G_0\cap Z=\{0\}$ and $G_0+Z=G+Z$. Let $f_1=g_1$ and $\nu_1 \in C_p(X)^*$ with $\nu_1|Z=0$ such that $\nu_1(f_1)=1.$

Assume that for some $n\in \IN$ we have chosen $$(f_1, \nu_1), \ldots, (f_n, \nu_n) \in C_p(X)\times C_p(X)^*$$ such that $$\mbox{lin}\{f_1, \ldots, f_n\}=\mbox{lin}\{g_1, \ldots, g_n\}$$ and $$\nu_j|Z=0,\;\; \nu_j(f_i)=\delta_{j,i}\mbox{ \  for all  }i,j\in \{1, \ldots, n\}.$$

Put $$f_{n+1}=g_{n+1}-\sum_{i=1}^n \nu_i(g_{n+1})f_i.$$ Then $$\mbox{lin}\{f_1, \ldots, f_{n+1}\}=\mbox{lin}\{g_1, \ldots, g_{n+1}\}$$ and $\nu_j(f_{n+1})=0$ for $1\leq j \leq n.$ Let $\nu_{n+1}\in C_p(X)^*$ with $\nu_{n+1}|Z=0$ such that $\nu_{n+1}(f_i)=0$ for $1\leq i \leq n$ and $\nu_{n+1}(f_{n+1})=1.$

Continuing on this  way we can construct inductively a biorthogonal sequence $\big((f_n, \nu_n)\big)_{n\in\mathbb N}$ in $C_p(X)\times C_p(X)^*$ such that $\mbox{lin} \{f_n: n\in \IN\}= \mbox{lin} \{g_n: n\in \IN\}$ and $\nu_n|Z=0$, $\nu_n(f_m)=\delta_{n,m}$ for all $n,m \in \IN.$ Then $\mbox{lin} \{f_n: n\in \IN\} +Z$ is dense in $C_u(X)$. Let $\mu_n=\nu_n/\|\nu_n\|$ for $ n\in \IN.$ Then $\|\mu_n\|=1$ and $\mu_n(f_m)=0$ for all $n,m \in \IN$ with $n\neq m.$

We  prove that $\mu_n(f)\to_n 0$ for every $f\in C_p(X).$ Given any $f\in C(X)$ and $\varepsilon>0$, find $m\in \IN$ and $g\in \mbox{lin}\{f_1, \ldots, f_m\} +Z$ with $d(f,g)<\varepsilon;$ clearly $d(f,g)=\|f-g\|_{\infty}.$ Then $\mu_n(g)=0$ for $n>m,$ so $$|\mu_n(f)|=|\mu_n(f-g)| \leq \|\mu_n\|{\cdot}\|f-g\|_{\infty}<\varepsilon$$ for $n>m.$ Thus $\mu_n(f)\to_n 0$, which means that the space $C_p(X)$ has the JNP.
\end{proof}


\section{An example of Plebanek}

In this section we describe the following example suggested to the authors by Grzegorz Plebanek \cite{plebanek}.

\begin{example}[Plebanek]\label{exa}
There exists a compact Hausdorff space $K$ such that
\begin{enumerate}
\item $K$ contain no nontrivial converging sequences but contains a copy of $\beta\mathbb{N}$;
\item the function space $C_p(K)$ has the JNP.
\end{enumerate}
\end{example}

We need some facts to present the construction of the space $K$. By definition, the {\em asymptotic density}  of a subset $A\subset \mathbb N$ is the limit $$d(A):=\lim_{n\to\infty}\frac{A\cap[1,n]}{n}$$if this limit exists. The family $\mathcal Z=\{A\subset\mathbb N:d(A)=0\}$ of sets of asymptotic density zero in $\mathbb N$ is an ideal on $\mathbb N$. Recall the following standard fact (here $A\subset^* B$ means that $A\setminus B$ is finite).
\medskip

\textbf{Fact 1}: \label{l1}
\emph{For any countable subfamily $\mathcal C\subset\mathcal{Z}$  there is a set $B\in\mathcal{Z}$ such that $C\subset^* B$ for all $C\in\mathcal C$.}
\medskip

Let $\fA=\big\{A\subset \mathbb N:d(A)\in\{0,1\}\big\}$ be the algebra of subsets of $\mathbb{N}$ generated by $\mathcal{Z}$.
We now let $K$ be the Stone space of the algebra $\fA$ so we treat elements of $K$ as ultrafilters on $\fA$. There are three types of such $x\in K$:
\begin{enumerate}
\item $\{n\} \in x$ for some $n\in\mathbb{N}$; then $x=\{A\in\fA: n\in A\}$ is identified with $n$;
\item $x$ contains no finite subsets of $\mathbb{N}$ but $Z\in x$ for some $Z\in\cZ$;
\item $Z\notin x$ for every $Z\in\cZ$; this defines the unique
\[p=\{A\in\fA:d(A)=1\}\in K.\]
\end{enumerate}
To see that $K$ is the required space it is enough to check the following  two facts.
\medskip

\textbf{Fact 2}. \label{l2}
\emph{The space $K$  contains no nontrivial converging sequence.}
\begin{proof}
 In fact we check that every infinite $X\subset K$ contains  an infinite set $Y$ such that $\overline{Y}$ is homeomorphic to $\beta\en$.
Note first that for every $Z\in\cZ$, the corresponding clopen set
\[\widehat{Z}=\{x\in K: Z\in x\},\]
is homeomorphic to $\beta\en$ because $\{A\in\fA: A\subset Z\}=P(Z)$.

For an infinite set $X\subset K$, we have two cases:

\medskip

\emph{Case 1},  $X\cap\en$ is infinite. There is an infinite $Z\subset X\cap\en $ having density zero. Then every subset of $Z$ is in $\fA$, which implies that
 $\overline{Z}\cong\beta\en$ .

\medskip

\emph{Case 2},  $X\cap ({K\sm \en})$ is infinite. Let us fix a sequence of different $x_n \in X\cap (K\sm\en)$ such that $x_n\neq p$ for every $n$.
 Then for every $n$ we have $Z_n\in x_n$ for some $Z_n\in\cZ$. Take $B\in\cZ$ as in Fact 1.
 Then $B\in x_n$  because $x_n$ is a nonprincipial ultrafilter on $\mathfrak A$ so $A_n\sm B\notin x_n$. Again, we conclude that $\overline{\{x_n: n\in\en\}}$ is $\beta\en$.
\end{proof}
\textbf{
Fact 3}: \label{l3}
\emph{If $\nu_n=\frac1n\sum_{k\le n} \delta_k$ and $\mu_n=\frac12(\nu_n - \delta_p)$ for $n \in \mathbb{N}$, then $\nu_n(f)\to_n \delta_p(f)$ and $\mu_n(f)\to_n 0$ for every $f\in C(K)$.}
\begin{proof}
Observe $\nu_n(A)\to_n d(A)$ for every $A\in\fA$ since elements of $\fA$ have asymptotic density either 0 or 1. This
means that, when we treat $\nu_n$ as measures on $K$ then $\nu_n(V)$ converges to $\delta_p(V)$ for every
clopen set $V\subset K$. This implies the assertion since every continuous function on $K$ can be uniformly approximated by
simple functions built from clopens.
\end{proof}

\section{Proof of Theorem~\ref{t:CP}}\label{s:CP}

Let us recall that a topological space $X$ is called
\begin{itemize}
\item {\em submetrizable} if $X$ admits a continuous metric;
\item {\em hemicompact} if $X$ has a countable family $\mathcal K$ of compact sets such that each compact subset of $X$ is contained in some compact set $K\in\K$;
\item a {\em $k$-space} if a subset $F\subset X$ is closed if and only if for every compact subset $K\subset X$ the intersection $F\cap K$ is closed in $K$.
\end{itemize}

In order to prove Theorem~\ref{t:CP}, we should check the equivalence of the following conditions for every Tychonoff space $X$:
\begin{enumerate}
\item $X$ is a submetrizable hemicompact $k$-space;
\item $C_k(X)$ is Polish;
\item  $C_p(X)$ is Polishable.
\end{enumerate}

$(1)\Rightarrow(2)$: If $X$ is a submetrizable hemicompact $k$-space, then $X=\bigcup_{n\in\w}X_n$ for some increasing sequence $(X_n)_{n\in\w}$ of compact metrizable spaces such that each compact subset of $X$ is contained in some compact set $X_n$. Then the function space $C_k(X)$ is Polish, being topologically isomorphic to the closed subspace $$\{(f_n)_{n\in\w}\in \prod_{n\in\w}C_k(X_n):\forall n\in\w\;\;f_{n+1}{\restriction}X_n=f_n\}$$ of the countable product $\prod_{n=1}^\infty C_k(X_n)$ of separable Banach spaces.
\smallskip

$(2)\Ra(1)$ If the function space $C_k(X)$ is Polish, then by Theorem 4.2 in \cite{McCoy}, $X$ is a hemicompact $k$-space.   Taking into account that the space $C_p(X)$ is a continuous image of the  space $C_k(X)$, we conclude that  $C_p(X)$ has countable network and by \cite[I.1.3]{Arhangel}, the space $X$ has countable network. By  \cite[2.9]{Gru}, the space $X$ is submetrizable.
\smallskip

The implication $(2)\Ra(3)$ follows from the continuity of the identity map $C_k(X)\to C_p(X)$.
\vskip3pt

$(3)\Ra(2)$: Assume that the space $C_p(X)$ is Polishable and fix a stronger Polish  locally convex topology $\tau$ on $C_p(X)$. Let $C_\tau(X)$ denote  the separable Fr\'echet space  $(C_p(X),\tau)$. By $\tau_{k}$ denote the compact open topology of $C_{k}(X)$. Taking into account that the space $C_p(X)$ is a continuous image of the Polish space $C_\tau(X)$, we conclude that  $C_p(X)$ has countable network and by \cite[I.1.3]{Arhangel}, the space $X$ has countable network and hence is Lindel\"of. By the normality (and the Lindel\"of property) of $X$,  each closed bounded set in $X$ is countably compact (and hence compact). So $X$ is a $\mu$-space. By Theorem 10.1.20 in \cite[Theorem 10.1.20]{bonet} the function space $C_{k}(X)$ is barrelled. The continuity of the identity maps $C_k(X)\to C_p(X)$ and $C_\tau(X)\to C_p(X)$ implies that the identity map $C_k(X)\to C_\tau(X)$ has closed graph.
Since $C_k(X)$ is barelled and $C_\tau(X)$ is Fr\'echet, we can apply the Closed Graph Theorem 4.1.10 in \cite{bonet} and conclude that the identity map $C_k(X)\to C_\tau(X)$ is continuous.

Next, we show that the identity map $C_\tau(X)\to C_k(X)$ is continuous. Given any compact set $K\subset X$ and any $\e>0$ we have to find a neighborhood $U\subset C_\tau(X)$ of zero such that $$U\subset \{f\in C(X):f(K)\subset(-\e,\e)\}.$$  The continuity of the restriction operator $R:C_p(X)\to C_p(K)$, $R:f\mapsto f{\restriction}K$, and the continuity of the idenity map $C_\tau(X)\to C_p(X)$ imply that the restriction operator $R:C_\tau(X)\to C_p(K)$ is continuous and hence has closed graph. The continuity of the identity map $C_k(K)\to C_p(K)$ implies that $R$ seen as an operator $R:C_\tau(X)\to C_k(K)$ still has closed graph. Since the spaces $C_\tau(X)$ and $C_k(K)$ are Fr\'echet, the Closed Graph Theorem 1.2.19 in \cite{bonet} implies that the restriction operator $R:C_\tau(X)\to C_k(K)$ is continuous. So, there exists a neighborhood $U\subset C_\tau(X)$ of zero such that $$R(U)\subset \{f\in C_k(K):f(K)\subset(-\e,\e)\}.$$  Then $U\subset\{f\in C(X):f(K)\subset(-\e,\e)\}$ and we are done. Hence $\tau=\tau_{k}$ is a Polish locally convex topology as claimed.

\end{document}